\documentclass[lang = american]{ems-icm-test} 
\newtheorem{theorem}{Theorem}
\newtheorem{proposition}[theorem]{Proposition}
\newtheorem{lemma}[theorem]{Lemma}
\newtheorem{corollary}[theorem]{Corollary}

\theoremstyle{definition}

\newtheorem{remark}[theorem]{Remark}

\newtheorem{example}[theorem]{Example}

\numberwithin{theorem}{section}

\makeatletter
\def\Ddots{\mathinner{\mkern1mu\raise\p@
\vbox{\kern7\p@\hbox{.}}\mkern2mu
\raise4\p@\hbox{.}\mkern2mu\raise7\p@\hbox{.}\mkern1mu}}
\makeatother

\newcommand{\PP}{\mathbb{P}}
\newcommand{\RR}{\mathbb{R}}

\newcommand{\CC}{\mathbb{C} }
\newcommand{\ZZ}{\mathbb{Z}}

\numberwithin{equation}{section}

\begin{document}

\volumetitle{}

\title{Beyond Linear Algebra}
\titlemark{Beyond Linear Algebra}

\emsauthor{1}{Bernd Sturmfels}{B.~Sturmfels}

\emsaffil{1}{Max-Planck Institute for Mathematics in the Sciences, \\
Inselstrasse 22, 04103 Leipzig, Germany
\email{bernd@mis.mpg.de} \qquad and \\
University of California, Berkeley CA 94720, USA \email{bernd@berkeley.edu}}


\begin{abstract}
Our title challenges the reader to venture beyond linear
algebra in designing models and in thinking about numerical algorithms for
identifying solutions.
This article accompanies the author's lecture at the
International Congress of Mathematicians 2022. 
It covers recent advances
in the study of critical point equations in optimization and statistics,
and it explores the role of nonlinear algebra in the study of linear~PDE with constant coefficients.
 \end{abstract}

\maketitle


\section{Introduction}

Linear algebra is ubiquitous in the mathematical universe.
It plays a foundational role for many models in the
sciences and engineering, and its numerical methods
are a driving force behind today's technologies.
 The power of linear algebra
stems from our ability, honed through the practice of calculus, to approximate
nonlinear shapes by linear spaces.

Yet, the world is nonlinear. Nonlinear equations are a natural
ingredient in mathematical models for the real world.
In our view, the true nonlinear nature of a phenomenon
should be respected as long as possible.
We argue against the common practice of passing to
a linear approximation immediately.
Of course, in the final step of implementing scalable 
algorithms, one will always employ the powerful tools
of numerical linear algebra. However, in the early phase of
exploring and designing a model, there is significant benefit in
going beyond linear algebra. Mathematical fields such as
algebraic geometry, algebraic topology, combinatorics,
commutative algebra or  representation theory furnish  practical tools.

The growing awareness of theoretical mathematics
in applications has led to a new field called
{\em Nonlinear Algebra}. The textbook \cite{MS} offers
foundations for interested students.
The aim of this lecture is to introduce research trends
and discuss a few recent results.
At the core of many problems 
lies the study of subsets of $\RR^n$ that are defined by polynomials:
\begin{equation}
\label{eq:semialgebraic}
\bigl\{\, x \in \RR^n:
f_1(x) = \cdots = f_k(x) = 0,
g_1(x) \geq 0, \ldots, g_l(x) \geq 0,
h_1(x) > 0 , \ldots, h_m(x) > 0 \bigr\}. 
\end{equation}
The set (\ref{eq:semialgebraic}) is a \emph{basic semialgebraic set}.
The Positivstellensatz \cite[Theorem 6.14]{MS}
gives a criterion for deciding whether this set is empty.
This seemingly theoretical criterion
has become a practical numerical method,
thanks to sums of squares \cite[\S 12.3]{MS}  and
semidefinite programming  \cite{BPT}.  
In addition to this, there are symbolic algorithms for
real algebraic geometry (cf.~\cite{ARS}).
So, the user has a wide range of choices
for working with semialgebraic~sets.

In this article we   disregard the inequalities in (\ref{eq:semialgebraic}) and
retain the equations only:
\begin{equation}
\label{eq:algebraic}
X \,\, = \,\, \bigl\{\, x \in \RR^n: f_1(x) = \cdots = f_k(x) = 0  \bigr\}. 
\end{equation}
This is a \emph{real algebraic variety}. We wish to answer questions about $X$
by reliable numerical computations, in particular using
tools such as \texttt{Bertini} \cite{Bertini}
or \texttt{HomotopyContinuation.jl} \cite{HCjl}.
We focus on questions that are addressed by solving
auxiliary polynomial systems with finitely many solutions,
where the number of complex solutions can be determined a priori.

In Section \ref{sec2} that number is the
Euclidean distance degree (ED degree) of $X$.
This governs the following question:
given $u \in \RR^n \backslash X$, which point in $X$
is nearest to $u$ in Euclidean distance?
We derive the critical equations of this optimization problem (\ref{eq:EDproblem}),
and we consider all solutions to these equations, both real and complex.
These include all local minima and local maxima.
Theorem \ref{thm:polarED} expresses
the ED degree in terms of the polar degrees of $X$.
Knowing these invariants allows us to 
find all critical points numerically, along with
a proof of correctness \cite{BRT}.
We ask our nearest point question also for other norms,
notably those given by a polytope.
The polar degrees appear again,
in Proposition~\ref{prop:polari}.

Section \ref{sec3} concerns algebraic varieties $X$
that serve as models in statistics. Their points
represent probability distributions. We focus on models for
Gaussian distributions and discrete
distributions. In these two scenarios, the ambient space
$\RR^n$ in (\ref{eq:algebraic}) is replaced
by the positive-definite cone $\text{PD}_n$
and by the probability simplex $\Delta_n$.
Given any data set, we wish to ascertain whether $X$ is
an appropriate model. To this end, 
maximum likelihood estimation (MLE) is used.
This optimization problem is stated explicitly in (\ref{eq:loglikegaussian2}) 
and (\ref{eq:MLE}), and we employ nonlinear algebra \cite{MS} in addressing it.
The number of complex critical points
is the maximum likelihood degree (ML degree) of the model~$X$.
Theorem \ref{thm:Euler} relates this to the Euler characteristic of
the underlying very affine variety. We apply this theory to a class of models
arising in particle physics, namely the configuration space of $m$ labeled points in 
general position in $\PP^{k-1}$. Known ML degrees for these models are given in
Theorem \ref{thm:CEGM}.

In Section \ref{sec4},
we turn to an analytic interpretation
of the polynomial system in (\ref{eq:algebraic}).
The unknowns $x_1,\ldots,x_n$ are replaced
by differential operators $\frac{\partial}{\partial z_1},\ldots,\frac{\partial}{\partial z_n}$,
and the polynomials $f_1,\ldots,f_k$ are viewed
as linear partial differential equations (PDE) with constant coefficients.
The variety $X$ is replaced by the space of
functions $\phi(z_1,\ldots,z_n)$ that are solutions to the PDE.
That space is typically infinite-dimensional. Our task is to compute it.
Algorithms are based on differential primary decompositions
\cite{AHS, CS, CHS}.
We also study linear PDE for vector-valued functions.
These are expressed by modules over a polynomial ring.

This article accompanies a lecture to be given in July 2022 at the
International Congress of Mathematicians in St.~Petersburg.
It encourages mathematical scientists to employ polynomials
 in designing models and in thinking about numerical algorithms.
Sections \ref{sec2} and \ref{sec3} are concerned with
critical point equations in optimization and statistics.
Section \ref{sec4} offers a glimpse on how
 nonlinear algebra interfaces with the study of linear~PDE.

\section{Nearest Points on Algebraic Varieties}
\label{sec2}

We consider a model $X$  that is given as the zero set
in $\RR^n$ of a collection $\{f_1,\ldots,f_k\}$ of nonlinear polynomials 
in $n$ unknowns $x_1,\ldots,x_n$.  Thus, $X$ is a real algebraic variety.
We assume that $X$ is irreducible, that
$I_X = \langle f_1,\ldots,f_k\rangle$ is its prime ideal, and that
the set of nonsingular real points is  Zariski dense in $X$.
The $k \times n$ Jacobian matrix $\mathcal{J} = (\partial f_i / \partial x_j)$
has rank at most $c$ at any point $x \in X$, where $c =  \text{codim}(X)$, and
 $x$ is {\em nonsingular} on $X$ if the rank is exactly $c$.
 Explanations of these hypotheses are found in Chapter 2 of the textbook~\cite{MS}.

The following optimization problem arises in many applications.
Given a data point $u \in \RR^n \backslash X$, 
compute the distance to the model $X$. Thus, we
seek a point $x^*$ in $X$ that is closest to $u$.
The answer depends on the chosen metric. One might choose the
Euclidean distance,  a $p$-norm \cite{KKS}, or polyhedral norms, such as those
arising in optimal transport \cite{wasserstein}.
In all of these cases, the solution $x^*$ can be found by solving
a system of polynomial equations.

We begin by discussing the {\em Euclidean distance (ED) problem}, which is as follows:
\begin{equation}
\label{eq:EDproblem}
 \text{minimize} \,\, \sum_{i=1}^n (x_i - u_i)^2 \,\, \text{subject to} \,\,\, x \in X . 
 \end{equation}
 We now derive the critical equations for (\ref{eq:EDproblem}).
  The {\em augmented Jacobian matrix} $\mathcal{AJ} $ is the $(k+1) \times n$ matrix
 obtained by placing the row $(x_1-u_1,\ldots,x_n-u_n)$ atop
 the Jacobian matrix $\mathcal{J}$. We form the ideal
 generated by its $(c+1) \times (c+1)$ minors, we add the ideal of the model $I_X$, and
  we then saturate \cite[(2.1)]{DHOST} that sum by the ideal of $c \times c$ minors of $\mathcal{J}$.
 The result is the {\em critical ideal} $\mathcal{C}_{X,u}$ of the
 model $X$ with respect to the data $u$. The variety of
 $\mathcal{C}_{X,u}$ is the set of critical points of (\ref{eq:EDproblem}).
For random data $u$, this variety is finite and it contains
the optimal solution $x^*$, provided the latter is attained at a nonsingular point of $X$.

The algebro-geometric approach to the ED problem was pioneered in a project with
 Draisma, Horobe\c t, Ottaviani and Thomas \cite{DHOST}.
That article introduced
the {\em ED degree} of $X$. This is the cardinality of the complex 
algebraic variety in $\CC^n$ defined by the critical ideal $\mathcal{C}_{X,u}$.
The ED degree of a model $X$ measures the difficulty
of solving the ED problem for~$X$.

\begin{example}[Space curves] \label{ex:spacecurves}
Fix $n=3$ and let $X$ be the curve in $\RR^3$ defined by two general
polynomials $f_1$ and $f_2$ of degrees $d_1$ and $d_2$ in $x_1,x_2,x_3$.
The augmented Jacobian matrix is
\begin{equation}
\label{eq:AJ1}
 \mathcal{AJ} \,\, = \,\,\,\, \begin{pmatrix} 
x_1-u_1 && x_2 - u_2 && x_3 - u_3 \\
\partial f_1 / \partial x_1 &\,\,\,& \partial f_1 / \partial x_2 &\,\,\,&\partial f_1 / \partial x_3 \\
\partial f_2 / \partial x_1 &\,\,\,& \partial f_2 / \partial x_2 &\,\,\,&\partial f_2 / \partial x_3 \\
\end{pmatrix}.
\end{equation}
For random data $u \in \RR^3$,
the ideal $\,\mathcal{C}_{X,u} =  \bigl\langle f_1,\, f_2 ,\, \text{det}(\mathcal{AJ}) \bigr\rangle\,$
has $d_1 d_2 (d_1+d_2-1)$ zeros in $\CC^3$,
by B\'ezout \cite[Theorem 2.16]{MS}. Hence the ED degree of 
$X$ equals $d_1 d_2 (d_1+d_2-1)$.
This can also be seen using the general formula from algebraic geometry in
\cite[Corollary 5.9]{DHOST}. If $X$ is a general smooth curve 
of degree $d$ and genus $g$, then
$\text{EDdegree}(X) = 3d + 2g-2$.
The above curve in $3$-space has degree
$d = d_1 d_2$ and genus $g = d_1^2d_2/2 + d_1 d_2^2/2 - 2 d_1 d_2 + 1    $.
\end{example}

Here is a general  upper bound on the ED degree
in terms of the given polynomials.

\begin{proposition} \label{prop:implgeneric}
Let $X $ be a variety of codimension $c$ in $\RR^n$ whose ideal $I_X$ is generated by
polynomials $f_1,f_2, \ldots,f_c, \ldots , f_k\,$ of degrees
$\,d_1\ge d_2\ge\cdots \ge d_c\ge\cdots \ge d_k$.
Then
\begin{equation}
\label{eq:EDdegCI}
 \text{EDdegree}(X) \,\,\,\le \,\,\,\,
d_1 d_2 \cdots d_c \cdot \!\!\!\!\!\!\!\!\!\!\!\!
\sum_{i_1+i_2+\cdots+i_c \leq n-c} \!\!\! \!\! (d_1-1)^{i_1} (d_2-1)^{i_2} \cdots (d_c-1)^{i_c}.
\end{equation}
Equality holds when $X$ is a generic complete intersection of codimension $c$
(hence $c=k$).
\end{proposition}

This appears in \cite[Proposition 2.6]{DHOST}. We can derive it as follows.
B\'ezout's Theorem ensures that the degree of the variety $X$ is at most $d_1 d_2 \cdots d_c$.
The entries in the $i$th row of the matrix $\mathcal{AJ}$ are polynomials of degrees $d_i-1$.
The degree of the variety of $(c+1) \times (c+1)$ minors of  $\mathcal{AJ}$ is at most
the sum in (\ref{eq:EDdegCI}). The intersection of that variety with $X$ is our set of
critical points, and the cardinality of that set is bounded by the product of the two degrees.
Generically, that intersection is a complete intersection  and the inequality (\ref{eq:EDdegCI}) is attained.

Formulas or a priori bounds for the ED degree are important when studying
 exact solutions to the optimization problem (\ref{eq:EDproblem}).
The paradigm is to compute all complex critical points,
 by either symbolic or numerical methods, and to then
extract one's favorite real solutions among these.
 This leads, for instance, to all local minima 
in (\ref{eq:EDproblem}). The ED degree is an upper bound on the number of
real critical points, but this bound is generally not tight. 

\begin{example} Consider the case $n=2,c=1,d_1 = 4$ in Proposition
\ref{prop:implgeneric}, where $X$ is a quartic curve in the
plane $\RR^2$. The number of complex critical points is $\textrm{EDdegree}(X) = 16$.
But, they cannot be all real.
For an illustration, consider the  {\em Trott curve}
$X = V(f)$, defined by
$$ f \,\,=\,\,  144(x_1^4+x_2^4) \,- \, 225(x_1^2+x_2^2) \, + \, 350x_1^2 x_2^2 \, + \, 81 . $$

\begin{figure}[h]
                \vspace{-0.06cm}
                \begin{center} 
                        \includegraphics[scale=0.31]{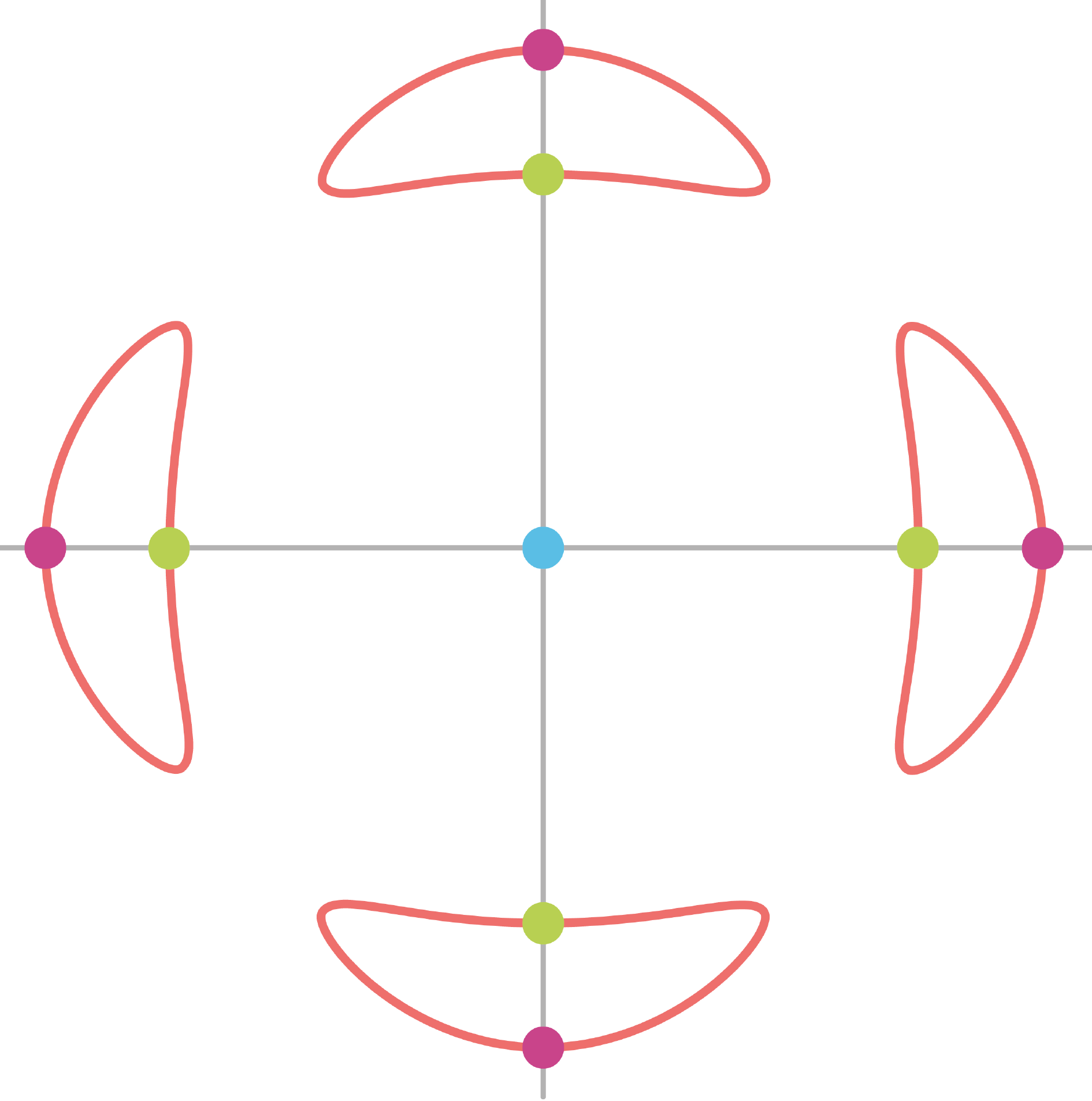}
                        \quad 
			\includegraphics[scale=0.32]{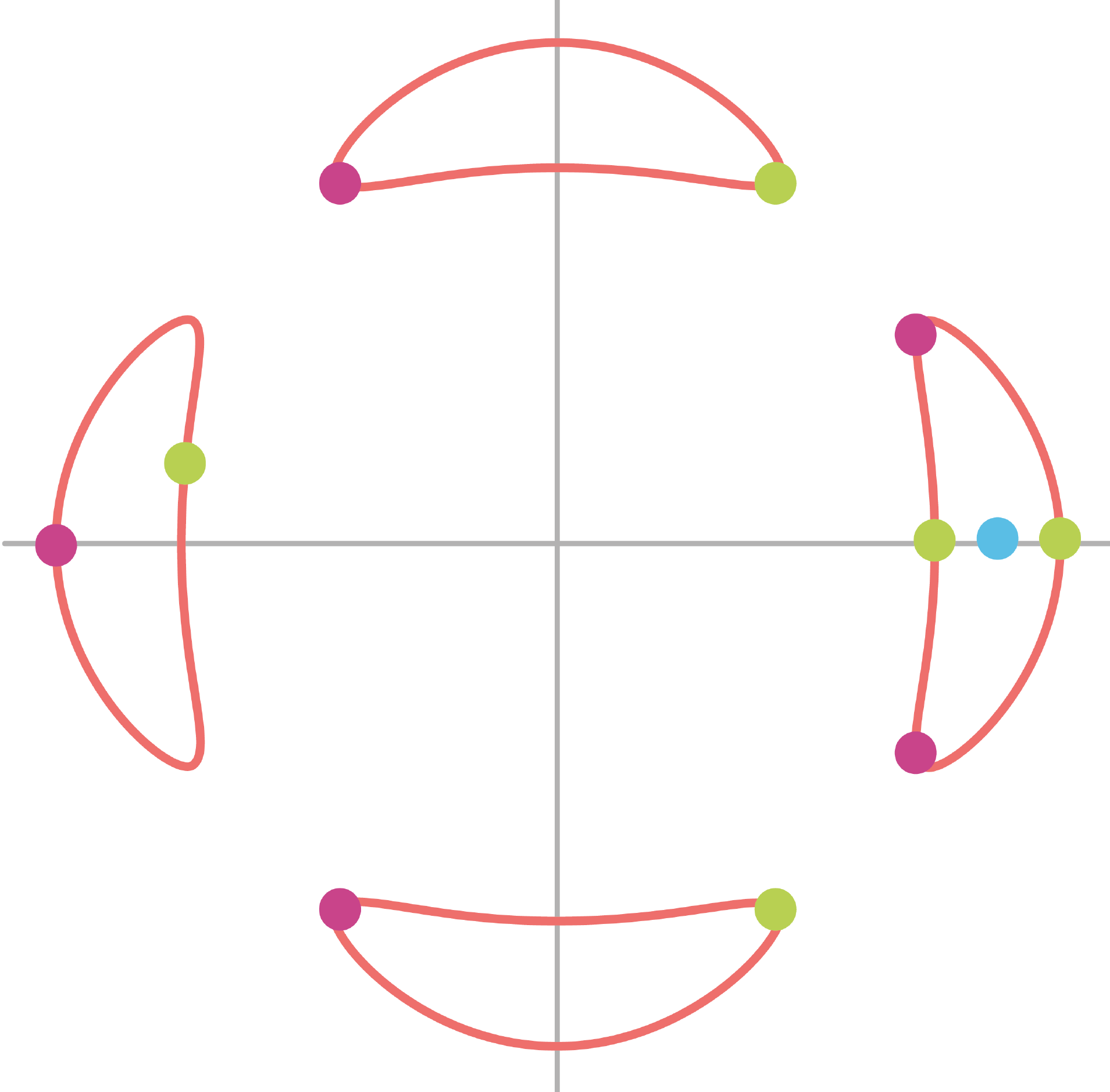}
                \end{center}
                \vspace{-0.5cm}
                \caption{
ED problems on  the Trott~curve:
  configurations of eight (left) or ten (right) critical points.              
                \label{fig:trottcurves}}
        \end{figure}
For general data $u = (u_1,u_2)$ in $\RR^2$,
we find $16$ complex solutions to the critical equations
$f = \frac{\partial f}{\partial x_2} (x_1 - u_1) - \frac{\partial f}{\partial x_1} (x_2-u_2) = 0$.
 For $u$ near the origin,
eight of them are real. 
For $u = \bigl(\frac{7}{8},\frac{1}{100}\bigr)$, which is inside the rightmost oval,
there are $10$ real critical points.
The two scenarios are shown in Figure \ref{fig:trottcurves}.
Local minima are green, while local maxima are purple.
For $u = (2,\frac{1}{100})$, to the right of the rightmost oval, the number
of real critical points is $12$.
\end{example}

In general, our task is to compute the zeros of the critical ideal $C_{X,u}$.
Algorithms for this computation
can be either symbolic or numerical. Symbolic methods
usually rest on the construction of a Gr\"obner basis, to be
followed by a floating point computation to extract the 
solutions. In recent years,
numerical methods have become increasingly popular. 
These are based on homotopy continuation. 
Two notable packages are
\texttt{Bertini} \cite{Bertini} and \texttt{HomotopyContinuation.jl} \cite{HCjl}.
The ED degree is important here because it indicates
how many paths need to be tracked to solve (\ref{eq:EDproblem}).
We next illustrate current capabilities.

\begin{example}
\label{ex:simon}
Suppose $X$ is defined by $c=k=3$ random
polynomials in $n=7$ variables, for a range of degrees $d_1,d_2,d_3$.
The table below lists  the ED degree in each case, and
the times used by \texttt{HomotopyContinuation.jl}
to compute and certify all critical points in~$\CC^7$.
$$
\begin{matrix} d_1 \,d_2 \, d_3 &
\quad 3 \, 2\, 2  \quad & \quad
3 \, 3 \, 2   \quad & \quad
3 \, 3 \, 3  \quad &   \quad
4 \, 2 \, 2  \quad &  \quad 
4\, 3\, 2  \quad &  \quad
4\, 3\, 3  \quad &  \quad
4\, 4\, 2  \quad &  \quad
4\, 4\, 3 \quad 
\\\ \textrm{EDdegree} &
\quad 1188 &
\quad 3618 &
\quad 9477 &
\quad 4176 &
\quad 10152 &
\quad 23220 &
\quad 23392 &
\quad 49872 \\
\textrm{Solving (sec)} &
\quad 3.849 &
\quad 21.06 &
\quad 61.51 &
\quad 31.51 &
\quad 103.5 &
\quad 280.0 &
\quad 351.5 &
\quad 859.3 \\
\!\textrm{Certifying (sec)} \!\!\!\! &
\quad 0.390 &
\quad 1.549 &
\quad 4.653 &
\quad 2.762 &
\quad 7.591 &
\quad 17.16 &
\quad 21.65 &
\quad 50.07 \\
\end{matrix}
$$
Here we represent $C_{X,u}$
by a system of $10$ equations in $10$ variables.
In addition to the three equations $f_1=f_2=f_3 = 0$
in $x_1,\ldots,x_7$, we take the seven equations
$(1,y_1,y_2,y_3) \cdot \mathcal{AJ} = 0$.
Here $y_1,y_2,y_3$ are new variables.
These ensure that the
$4 \times 7$ matrix $\mathcal{AJ}$ has rank $\leq 3$.
In all cases the timings include the certification step~\cite{BRT}
that proves correctness and completeness. 
These computations were performed using \texttt{HomotopyContinuation.jl} v2.5.6 on a 16 GB MacBook Pro with an Intel Core i7 processor working at 2.6 GHz.
They suggest that our critical equations can be solved fast and reliably, with proof
of correctness, when the ED degree is less than $50000$. 
For even larger numbers of solutions, success with numerical path tracking
will depend on the specific structure of the problem. If the discriminant
is well-behaved, then larger ED degrees are feasible.
An example of this appears in \cite[Table 1]{ST}.
\end{example}

We next present a general formula for ED degrees in terms
of projective geometry. 

\begin{theorem} \label{thm:polarED}
If $\,X$ meets both the hyperplane at infinity and the isotropic quadric transversally, then
$\text{EDdegree}(X)$ equals
 the sum of the polar degrees of the projective closure of~$X$.
\end{theorem}

The \emph{projective closure} of $X \subset \RR^n$
is its Zariski closure in complex projective space $\PP^n$,
which we will also denote by $X$. Theorem \ref{thm:polarED}
appears in \cite[Proposition 6.10]{DHOST}.
The hypothesis is stated precisely in \cite[equation (6.4)]{DHOST}.
It holds for all $X$ after a general linear change of coordinates.
We now explain what the polar degrees of a variety $X \subset \PP^n$ are. Points
$h $ in the dual projective space $(\PP^n)^\vee$
represent hyperplanes  $\{ x \in \PP^n: h_0 x_0 + \cdots + h_n x_n = 0 \}$.
We are interested in all pairs $(x,h)$ in $\PP^n \times (\PP^n)^\vee$ such that
$x $ is a nonsingular point of $X$ and $h$ is tangent to $X$ at $x$.
The Zariski closure of this set is the \emph{conormal variety} $N_X \subset \PP^n \times 
(\PP^n)^\vee$.

It is known that $N_X$ has dimension $n-1$, and  if $X$ is irreducible then so is $N_X$.
The image of $N_X$ under projection onto the second factor is the dual variety $X^\vee$.
The role of $x \in \PP^n$ and $h \in (\PP^n)^\vee$ can be swapped.
The following biduality relations  \cite[\S I.1.3]{GKZ} hold:
$$ N_X = N_{X^\vee} \quad \text{and} \quad (X^\vee)^\vee = X. $$
The class of $N_X$ in the cohomology ring 
$ H^*( \PP^n {\times} (\PP^n)^\vee\! ,\, \ZZ) =  \ZZ[s,t]/\langle s^{n+1},t^{n+1} \rangle $ 
has the~form
$$ [N_X]\, = \,
\delta_1(X) s^n t \,+\, \delta_2(X) s^{n-1} t^2 \,+\, \delta_3(X) s^{n-2} t^3  \,+\, \cdots \,+\, 
\delta_n(X) s t^{n}.
$$
The coefficients $\delta_i(X)$ of this binary form are nonnegative integers,
known as \emph{polar degrees}.

\begin{remark}
The polar degrees satisfy
$\delta_i(X)  = \#( N_X \,\cap \, (L \times L'))$, where
$L \subset \PP^n$ and $ L' \subset (\PP^n)^\vee$ are general linear subspaces
of dimensions $n+1-i\,$ and $\,i$ respectively.
This geometric interpretation implies that $\delta_i(X) = 0$ for
$i < \text{codim}(X^\vee)$ and for $i > \text{dim}(X)+1$.
\end{remark}

\begin{example}
Let $X$ be a general surface of degree $d$ in $\PP^3$.
Its dual $X^\vee$ is a surface of degree $d(d-1)^2$ in $(\PP^3)^\vee$.
The conormal variety $N_X$ is a surface in $\PP^3 \times (\PP^3)^\vee$, with class
$$  [N_X] \,\,\, = \,\,\, d(d-1)^2 \,s^3 t \,\,+\,\, d(d-1) \,s^2 t^2 \,\,+ \,\, d \,st^3. $$
The sum of the three polar degrees equals $\text{EDdegree}(X) = d^3 -d^2+d$;
see Proposition \ref{prop:implgeneric}.
\end{example}

Theorem \ref{thm:polarED} allows us to compute the ED degree
for many interesting varieties, e.g.~using Chern classes
\cite[Theorem 5.8]{DHOST}. This is relevant for applications in machine learning
\cite{Kohn} which rest on
low-rank approximation of matrices and tensors with special structure \cite{OPS}.

The discussion so far was restricted to the Euclidean norm.
But, we can measure distances in $\RR^n$ with any other
norm $|| \,\cdot\,||$. Our optimization problem 
(\ref{eq:EDproblem}) extends naturally:
\begin{equation}
\label{eq:normproblem}
 \text{minimize} \,\, || x - u || \,\, \text{subject to} \,\,\, x \in X . 
 \end{equation}
The unit ball $B = \{ x \in \RR^n : || x || \leq 1 \}$
is a centrally symmetric convex body. Conversely,~every
centrally symmetric
 convex body $B$ defines a norm, and we can 
paraphrase (\ref{eq:normproblem}) as follows:
\begin{equation}
\label{eq:normproblem2}
 \text{minimize} \,\,\lambda \,\, \text{subject to} \,\,\, \lambda \geq 0 \,\,\text{and} \,\,
 (u + \lambda B)  \, \cap \, X \,\,\not = \,\, \emptyset. 
 \end{equation}

If the boundary of $B$ is smooth and algebraic then we 
express the critical equations as a polynomial system.
This is derived as before, but we now replace
the first row of the augmented Jacobian matrix
$\mathcal{AJ}$ with the gradient of the map $\, \RR^n \rightarrow \RR, \,x \mapsto || x-u||$.

Another case of interest arises when 
$|| \cdot ||$ is a \emph{polyhedral norm}.
This means that $B$ is a centrally symmetric polytope.
Familiar examples of polyhedral norms
are $|| \cdot ||_\infty$ and $|| \cdot ||_1$, where
$B$ is the cube and the crosspolytope respectively.
In optimal transport theory, one uses 
a Wasserstein norm \cite{wasserstein} whose unit ball $B$ is
the polar dual of a Lipschitz polytope.

To derive the critical equations, a combinatorial
stratification of the problem is used, given by the face poset of the polytope $B$.
Suppose that $X$ is in general position. Then
$(u + \lambda^* B)  \, \cap \, X \, = \, \{x^*\}$ is a singleton
for the optimal value $\lambda^*$ in
   (\ref{eq:normproblem2}). 
The point  $\frac{1}{\lambda^*}(x^*-u)$ lies in the relative interior
of a unique face $F$ of the unit ball $B$.
   Let $L_F$ denote the linear span of $F$ in~$\RR^n$.
We have $   \text{dim}(L_F) = \text{dim}(F)+1$.
      Let $\ell$ be any linear functional on $\RR^n$ that attains its minimum over 
      the polytope $B$ at the face $F$. We view $\ell$ as a point in $(\PP^n)^\vee$.

\begin{lemma}
The optimal point   $x^*$ in (\ref{eq:normproblem}) is the unique solution to the  optimization problem
   \begin{equation}
   \label{eq:minell}  \text{Minimize} \,\,\, \ell(x)\,\,\, \text{subject to} \,\, x \in (u + L_F) \cap X.  
\end{equation}   
\end{lemma}

\begin{proof}
The general position hypothesis ensures that
$u+L_F$ intersects $X$ transversally, and $x^*$ is a smooth point of that intersection.
Moreover, $x^*$ is a minimum of the restriction of $\ell$ to the variety
$ (u + L_F) \cap X $. By our hypothesis, this linear function is generic
relative to the variety, so the number of critical points is finite and the function values are distinct.
\end{proof}

The problem (\ref{eq:minell}) amounts to linear programming over
a real variety.  We now determine the algebraic degree
of this optimization task when $F$ is a face of codimension $i$.

\begin{proposition} \label{prop:polari}
Let $L$ be a general affine-linear space of 
codimension $i-1$ in $\RR^n$ and $\ell$ a general linear form.
The  number of critical points of $\ell$ on $L \cap X$ is the polar
degree $\delta_i(X)$.
\end{proposition}

\begin{proof} This result is
\cite[Theorem 5.1]{wasserstein}.
The number of critical points of a linear form is the degree of the
dual variety $(L \cap X)^\vee$. That degree coincides with the
polar degree $\delta_i(X)$.
\end{proof}

\begin{figure}[h]
\vspace{-0.4cm}
                \begin{center}
                       \qquad \includegraphics[scale=1.3]{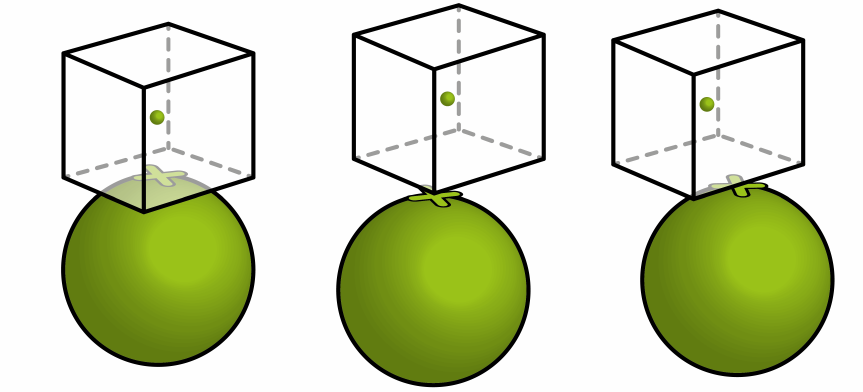}
                \end{center}
                \vspace{-0.6cm}
                \caption{
The cube is the $|| \cdot ||_\infty$ ball $\lambda^* B$
around the green point $u$. The variety $X$
is the sphere. The contact point $x^*$ is marked
with a cross. The optimal face $F$ is
a facet, a vertex, or an~edge.                                
                \label{fig:cubeball}}
        \end{figure}

\begin{example}
Consider (\ref{eq:normproblem}) and (\ref{eq:normproblem2}) where $X$ is 
a general surface of degree $d$ in $\RR^3$.  The optimal face $F$
of the unit ball $B$ depends on the location of the data point $u$.
This is shown for $d=2$ and $|| \cdot ||_\infty$  in Figure \ref{fig:cubeball}.
The algebraic degree of the solution $x^*$ equals
$\delta_3(X) = d$ if $\text{dim}(F)=0$,
it is $\delta_2(X) = d(d-1)$ if $\text{dim}(F)=1$,
and it is $\delta_1(X) = d(d-1)^2$ if $\text{dim}(F)=2$.
\end{example}

We conclude that the conormal variety $N_X$ and its cohomology class
$[N_X]$ are key players when it comes to  reliably solving the
distance minimization  problem for a  variety~$X $.
The polar degrees $\delta_i(X)$ reveal precisely how many paths
need to be tracked by numerical solvers like
\cite{Bertini, HCjl} in order to find
and certify \cite{BRT}  the optimal solution $x^*$ to
(\ref{eq:EDproblem}) or~(\ref{eq:normproblem}).

\section{Likelihood Geometry}
\label{sec3}

The previous section was concerned with  minimizing
the distance from a given data point $u$ to a model $X$
that is described by polynomial equations. In what follows we consider
the analogous problem in the setting of algebraic statistics \cite{Sul}, where
the model $X$ represents a family of probability distributions.
Distance to $u$ is replaced by the log-likelihood function.

The two scenarios of most interest for statisticians are
Gaussian models and discrete models.
We shall discuss them both, beginning with the Gaussian case.
Let $\text{PD}_n$ denote the open convex cone of positive-definite
symmetric $n \times n$ matrices.
Given a mean vector $\mu \in \RR^n$ and a covariance matrix $\Sigma \in \text{PD}_n$,
the associated {\em Gaussian distribution} on $\RR^n$ has the density
$$ f_{\mu,\Sigma}(x) \,\,\, := \,\,\, \frac{1}{\sqrt{(2 \pi)^n \,\text{det}\, \Sigma}} \cdot
\text{exp} \biggl( -\frac{1}{2} (x-\mu)^T \Sigma^{-1} (x-\mu) \biggr). $$
We fix a model $Y \subset \RR^n \times \text{PD}_n$
that is defined by polynomial equations in $(\mu,\Sigma)$.
Suppose we are given $N$ samples $U^{(1)}\!, \ldots, U^{(N)}$ in $\RR^n$.
These are summarized in the
\emph{sample~mean}
$\bar{U} = \frac{1}{N} \sum_{i=1}^N U^{(i)}$
and in the \emph{sample covariance matrix}
$S = \frac{1}{N} \sum_{i=1}^N (U^{(i)} - \bar{U})(U^{(i)} - \bar{U})^T$.
Given these data,
the log-likelihood  is the following function in the unknowns $(\mu,\Sigma)$:
\begin{equation}
\label{eq:loglikegaussian}
\ell(\mu,\Sigma) \quad = \quad - \frac{N}{2} \cdot \biggl[\,
\text{log}\, \text{det}\, \Sigma  \,\,+ \,\, \text{trace}(S \Sigma^{-1}) \,\, + \,\,
(\bar{U} - \mu)^T  \Sigma^{-1} (\bar{U} - \mu) \biggr].
\end{equation}
The task of likelihood inference is to minimize this function subject to
$(\mu,\Sigma) \in Y$.

There are two extreme cases. First, consider
a model where $\Sigma$ is fixed to be the identity matrix  $\text{Id}_n$. Then
$Y = X \times \{\text{Id}_n\}$ and we are supposed to minimize
the Euclidean distance from the sample mean $\bar{U}$
 to the variety $X$ in $\RR^n$. 
This is precisely our problem~(\ref{eq:EDproblem}).

We instead focus on the second case, the family of {\em centered Gaussians},
where $\mu$ is fixed at zero.
The model has the form
$\{0\} \times X$, where $X$ is a variety in the space
$\text{Sym}_2(\RR^n)$ of symmetric $n \times n$ matrices.
Following \cite[Proposition 7.1.10]{Sul},
 our task is now as follows:
\begin{equation}
\label{eq:loglikegaussian2}
\text{Minimize the function} \,\,\,
\,\Sigma \,\,\mapsto\,\,
\text{log}\, \text{det}\, \Sigma  \,\,+ \,\, \text{trace}(S \, \Sigma^{-1}) \quad
\text{subject to} \,\, \Sigma \in X. 
\end{equation}
Using the concentration matrix $K = \Sigma^{-1}$, we can write this equivalently as follows:
\begin{equation}
\label{eq:loglikegaussian3}
\text{Maximize the function} \,\,\,
\,\Sigma \,\,\mapsto\,\,
\text{log}\, \text{det}\, K \,\,- \,\, \text{trace}(S \, K) \quad
\text{subject to} \,\, K \in X^{-1}.
\end{equation}
Here the variety $X^{-1}$ is the Zariski closure of the set of inverses of all matrices in $X$.

The critical equations of the optimization problem (\ref{eq:loglikegaussian3})
can be written as polynomials, since the partial derivatives of
the logarithm are rational functions. These equations
have finitely many complex solutions. Their number is
the {\em ML degree} of the model $X^{-1}$.

Let $\mathcal{L} \subset \text{Sym}_2(\RR^n)$ be
a linear space of symmetric matrices (LSSM),
whose general element is assumed to be invertible.
We are interested in the models $X^{-1} = \mathcal{L}$
and $X = \mathcal{L}$.
It is convenient to use primal-dual coordinates
$(\Sigma,K)$ to write the respective critical equations.

\begin{proposition} Fix an LSSM $ \mathcal{L}$ and its
orthogonal complement  $\mathcal{L}^\perp$
for the inner product $\langle X,Y \rangle = \text{trace}(XY)$.
The critical equations for the \emph{linear concentration model} $X^{-1} = \mathcal{L}$~are
\begin{equation}
\label{eq:linconc} K \in \mathcal{L} \,\,\,\text{and} \,\,\,
K \Sigma = \text{Id}_n  \,\,\,\text{and} \,\,\,
\Sigma - S \in \mathcal{L}^\perp . 
\end{equation}
The critical equations for the \emph{linear covariance model} $X = \mathcal{L}$ are
\begin{equation}
\label{eq:lincova} \Sigma \in \mathcal{L} \,\,\,\text{and} \,\,\,
K \Sigma = \text{Id}_n  \,\,\,\text{and} \,\,\,
KSK - K \in \mathcal{L}^\perp . 
\end{equation}
\end{proposition}

\begin{proof}
This is well-known in statistics. For proofs see \cite[Propositions 3.1 and 3.3]{STZ}.
\end{proof}

The system (\ref{eq:linconc}) is linear in $K$, but
the last group of equations in (\ref{eq:lincova}) is quadratic in $K$.
The numbers of complex
solutions are the \emph{ML degree} of $\mathcal{L}$ and
the \emph{reciprocal ML degree} of $\mathcal{L}$.
The former is  smaller than the latter,
and (\ref{eq:linconc}) is easier to solve than (\ref{eq:lincova}).

\begin{example} \label{ex:siebzehn}
Let $n=4$ and $\mathcal{L}$ a generic LSSM of dimension $k$. Our degrees are as follows:
$$ 
\begin{matrix}
k = \text{dim}(\mathcal{L}): \qquad & \qquad 2 \quad & \quad 3 \quad & 
\quad 4 \quad & \quad 5 \quad & \quad 6 \quad & \quad 7 
\quad & \quad 8 \quad & \quad 9 
\\
\text{ML degree}: \qquad  & \qquad  3 \quad & \quad 9 \quad & \quad 17 \quad &
\quad  21 \quad & \quad 21 \quad & \quad 17 \quad & \quad 9 \quad & \quad 3 \quad 
\\
\text{reciprocal ML degree}:  \qquad & \qquad 5 \quad & \quad 19 \quad  & 
\quad 45 \quad & \quad 71 \quad & \quad 81 \quad  & \quad 63 \quad & \quad 
29 \quad  & \quad  7 
\\
\end{matrix}
$$
These numbers and many more appear in \cite[Table 1]{STZ}.
\end{example}

ML degrees and the reciprocal ML degrees have been studied intensively
in the recent literature, both for generic and special spaces $\mathcal{L}$.
See \cite{AGKMS, BCEMR, EFSS} and the references therein.
We now present an important result due to
Manivel, Micha{\l}ek, Monin, Seynnaeve, Vodi\v{c}ka  and Wi\'sniewski.
Theorem \ref{thm:schubertML} paraphrases highlights from their articles  \cite{MMMSV, MMW}.

\begin{theorem} \label{thm:schubertML}
The ML degree of a generic linear subspace $\mathcal{L}$ 
of dimension $k$
in $\text{Sym}_2(\RR^n)$ is the number
of quadrics in $\PP^{n-1}$ that pass through
$\binom{n+1}{2}-k$ general points and
are tangent to 
$k-1$ general
 hyperplanes.
 For fixed $k$, this number is a polynomial in $n$ of degree $k-1$. 
\end{theorem}

\begin{proof}
The first statement is
 \cite[Corollary 2.6 (4)]{MMW}, here interpreted classically
   in terms of Schubert calculus. 
For a detailed discussion see the introduction of \cite{MMMSV}.
The second statement appears in
 \cite[Theorem 1.3 and Corollary 4.13]{MMMSV}. It proves a
conjecture of Sturmfels and Uhler.
\end{proof}

\begin{example}[$n=4$]
Fix $10-k$ points
and $k-1$ planes in $\PP^3$. We seek 
quadratic surfaces containing the points and
tangent to the planes. This imposes $9$ constraints on
$\PP(\text{Sym}_2(\CC^4)) \simeq \PP^9$. Passing through a point
is a linear equation. Being tangent to a plane is a cubic equation.
B\'ezout's Theorem  suggests that there could be $3^{k-1}$ solutions.
This is correct for $k \leq 3$ but it overcounts for $k \geq 4$. Indeed,
in Example \ref{ex:siebzehn} we see
$17,21,21,\ldots$ instead of $27,81,243,\ldots$.
\end{example}

The intersection theory in \cite{MMW, MMMSV} leads to formulas
for the ML degrees of linear Gaussian models. From this we obtain
provably correct numerical
methods for maximum likelihood estimation. Namely,
after computing critical points as in \cite{STZ},
we can certify them as in \cite{BRT}.
Since the ML degree is known, one can be sure that
all solutions have been found.

\smallskip

We now shift gears and turn our attention to discrete statistical models.
We take the state space to be $\{0,1,\ldots,n\}$.
The role of the cone $\text{PD}_n$ is played by the probability simplex
\begin{equation}
\label{eq:simplex}
\Delta_n \,\,\, = \,\,\, \bigl\{\,
p = (p_0,p_1,\ldots,p_n) \in \RR^{n+1} \,:\,
p_0+p_1+\cdots+p_n = 1 \,\,\text{and}\,\,
p_0,p_1,\ldots,p_n > 0 \,\bigr\}.
\end{equation}
Our model is a subset $X$ of $\Delta_n$  defined
by polynomial equations. As before, for venturing beyond linear algebra,
we identify $X$ with its Zariski closure in complex projective space~$\PP^n$.

We shall present the algebraic approach to
 maximum likelihood estimation (MLE). See~\cite{CHKS, DMS, Huh14, HS, HKS, Sul}
 and references therein.
Suppose we are given $N$ i.i.d.~samples. These are summarized in
the data vector $u = (u_0,u_1,\ldots,u_n)$ where $u_i$ is the number of
times state $i$ was observed. Note that $N = u_0 + \cdots + u_n$.
The associated log-likelihood function equals
$$ \ell_u : \Delta_n \rightarrow \RR \, ,\,\, p \,\mapsto \,
u_0 \cdot \text{log}(p_0) \,+ \,
u_1 \cdot \text{log}(p_1) \,+ \,\cdots + \,u_n \cdot \text{log}(p_n). $$
Performing MLE for the model $X$ means solving the following optimization problem:
\begin{equation}
\label{eq:MLE} \text{Maximize} \,\, \ell_u(p) \,\,
\text{subject to} \,\, p \in X.
\end{equation}
The \emph{ML degree} of $X$ is the number of
complex critical points of (\ref{eq:MLE}) for generic data $u$.
The optimal solution is denoted $\hat p$ and called the \emph{maximum likelihood estimate}
for the data $u$.

The critical equations for (\ref{eq:MLE}) are similar to those of (\ref{eq:EDproblem}).
Let $I_X = \langle f_1,\ldots,f_k \rangle + \langle  p_0+p_1+ \cdots + p_n- 1 \rangle $
 be the  defining ideal of the model.
Let $\mathcal{J} = \bigl( \partial f_i /\partial p_j \bigr)$ denote the
Jacobian matrix of size $(k+1) \times (n+1)$, and set $c = \text{codim}(X)$.
The augmented Jacobian
$\mathcal{AJ}$ is obtained by prepending one more row, namely  the
gradient of the objective function 
$$ \nabla \ell_u \,\, = \,\, \bigl(\,u_0/p_0,\, u_1/p_1,\,\ldots\,,\,u_n/p_n\, \bigr) . $$
To obtain the critical equations, enlarge $I_X$ by the $c \times c$ minors of
the $(k+2) \times (n+1)$ matrix $\mathcal{AJ}$, then clear denominators,
and finally remove extraneous components by saturation.

\begin{example}[Space curves]
Let $n=3$ and $X$ the curve in $\Delta_3$ defined by
two general  polynomials $f_1$ and $f_2$ of degrees 
$d_1$ and $d_2$ in $p_0,p_1,p_2,p_3$. The augmented Jacobian matrix~is
\begin{equation}
\label{eq:AJ2}
 \mathcal{AJ} \,\, = \,\,\,\, \begin{pmatrix} 
u_0 / p_0 && u_1/p_1 && u_2 / p_2 && u_3 / p_3 \\
  1 && 1 && 1 && 1 \\
  \partial f_1 / \partial p_0 &\,\,\,\,&
\partial f_1 / \partial p_1 &\,\,\,\,& \partial f_1 / \partial p_2 &\,\,\,\,&\partial f_1 / \partial p_3 \\
\partial f_2 / \partial p_0 &\,\,\,\,&
\partial f_2 / \partial p_1 &\,\,\,\,& \partial f_2 / \partial p_2 &\,\,\,\,&\partial f_2 / \partial p_3 \\
\end{pmatrix}.
\end{equation}
Clearing denominators amounts to multiplying the $i$th column by $p_i$, so 
the determinant contributes a polynomial of degree $d_1+d_2+1$
to the critical equations.
Since the generators of $I_X$ have degrees $d_1,d_2,1$, we conclude that
the ML degree of $X$ equals $d_1d_2(d_1+d_2+1)$.
\end{example}

The following  MLE analogue to Proposition \ref{prop:implgeneric}
 is established in \cite[Theorem 5]{HKS}.

\begin{proposition} \label{prop:implgeneric2}
Let $X $ be a model of codimension $c$ in $\Delta_n$ whose ideal $I_X$ is generated by
polynomials $f_1,f_2, \ldots,f_c, \ldots , f_k\,$ of degrees
$\,d_1\ge d_2\ge\cdots \ge d_c\ge\cdots \ge d_k$.
Then
\begin{equation}
\label{eq:MLdegCI}
 \text{MLdegree}(X) \,\,\,\le \,\,\,\,
d_1 d_2 \cdots d_c \cdot \!\!\!\!\!\!\!\!\!\!\!\!
\sum_{i_1+i_2+\cdots+i_c \leq n-c} \!\!\! \!\! d_1^{i_1} d_2^{i_2} \cdots d_c^{i_c}.
\end{equation}
Equality holds when $X$ is a generic complete intersection of codimension $c$
(hence $c=k$).
\end{proposition}

We next present the MLE analogue to Theorem \ref{thm:polarED}.
The role of the polar degrees is now played by the
Euler characteristic.
Consider $X$ in the complex projective space $\PP^n$,
and let $X^o$ be the open subset of $X$ that is obtained by
removing $\bigl\{p_0 p_1 \cdots p_n (\sum_{i=0}^n p_i) = 0\bigr\}$.
We recall from \cite{Huh13, Huh14} that
a \emph{very affine variety} is a closed subvariety of an algebraic torus~$(\CC^*)^r$. 

\begin{theorem} \label{thm:Euler}
Suppose that the very affine variety  $X^o$ is non-singular.
The ML degree of the model $X$ equals the signed Euler characteristic
$(-1)^{\text{dim}(X)} \cdot \chi(X^o)$ of the manifold $X^o$.
\end{theorem}

\begin{proof}[Proof and Discussion]
This was proved with a further smoothness assumption in \cite[Theorem 19]{CHKS},
and in full generality in \cite[Theorem 1]{Huh13}. If $X^o$ is singular then
the Euler characteristic can be replaced by the
Chern-Schwartz-MacPherson class, as shown in \cite[Theorem`2]{Huh13}.
\end{proof}

Of special interest is the case when the ML degree is equal to one.
This means that the estimate $\hat p$ is a rational function of the data $u$.
Here are two examples where this happens.

\begin{example}[$n=3$]
The independence model for two binary random variables is a quadratic
surface $X$ in the tetrahedron $\Delta_3$. This model is described by the constraints
$$ \text{det} \begin{bmatrix} p_0 \, & \, p_1 \\                                                           
        p_2 \, & \, p_3 \end{bmatrix} \,= \, 0 \quad \text{and} \quad
        p_0+p_1+p_2+p_3 = 1 \quad \text{and} \quad
        p_0,p_1,p_2,p_3 > 0.
       $$
Consider data $u= \begin{small} \begin{bmatrix} u_0 \, & \, u_1 \\                                                           
        u_2 \, & \, u_3 \end{bmatrix}\end{small} $ of       
          \emph{sample size}  $ \,|u| = u_0{+}u_1{+}u_2{+}u_3$.
       The ML degree of the surface $X$
        equals one because the MLE $\hat p$ is a rational function of the data, namely
\begin{equation}      
\label{eq:rational1}     \begin{matrix}                                                                 
\hat p_0 \,=\, |u|^{-2} (u_0 {+} u_1)(u_0{+}u_2)  , \, \,\quad                                           
\hat p_1 \,=\, |u|^{-2} (u_0 {+} u_1)(u_1{+}u_3)  , \,    \\                                           
\hat p_2 \,=\, |u|^{-2} (u_2 {+} u_3)(u_0{+}u_2)  , \,   \, \quad                                        
\hat p_3 \,=\, |u|^{-2} (u_2 {+} u_3)(u_1{+}u_3).\,
\end{matrix}      
\end{equation}
In words, we multiply the
          row sums with the column sums in the
empirical distribution  $\frac{1}{|u|}u $.
\end{example}

\begin{example}[$n=2$] \label{ex:coins}
Given a biased coin, we perform the following experiment:
\emph{Flip a biased coin. If it shows heads, flip it again}. The outcome
is the number of heads: $0$, $1$ or $2$.

\begin{figure}[h]
\vspace{-0.45cm}
                \begin{center}
                        \includegraphics[scale=1]{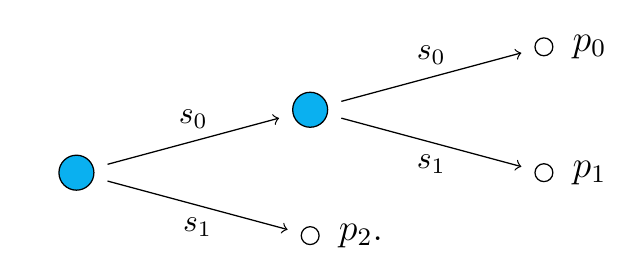}
                \end{center}
                \vspace{-0.7cm}
                \caption{
Probability tree that describes the coin toss model in Example \ref{ex:coins}.
                \label{fig:coins}}
        \end{figure}

If $s$ is the bias of the cone, then the model
is the parametric curve $X$ given by
   $$ \qquad (0,1) \to  X \subset \Delta_2\,, \,\,
  s \,\mapsto \, \bigr(\,s^2,s (1-s),1-s \,\bigr). $$
This model is the conic $   X = V(p_0p_2 - (p_0+p_1)p_1) \subset \PP^2$.
The MLE is given  by the formula
\begin{equation}      
\label{eq:rational2}
 (\hat p_0, \hat p_1, \hat p_2) \,=\,
\begin{small}                                                                              
\biggl(                                                                                    
\frac{(2u_0 + u_1)^2}{(2u_0 {+} 2u_1 {+} u_2)^2} \,,\,                                     
\frac{(2u_0 {+} u_1)(u_1{+}u_2)}{(2u_0 + 2u_1 + u_2)^2}\, ,\,                              
 \frac{u_1+u_2}{2u_0 {+} 2u_1 {+} u_2} \biggr) .                   
 \end{small}                                                                               
\end{equation}
Since the coordinates of $\hat p$ are rational functions,
the ML degree of $X$ is equal to one.
\end{example}

The following theorem explains what we saw in equations (\ref{eq:rational1})
and (\ref{eq:rational2}):

\begin{theorem}
If $X \subset \Delta_n$ is a model of ML degree one, so
 $\hat p$ is a rational function of $u$,
then each coordinate $\hat p_i$ is an alternating  product of linear forms with positive coefficients.
\end{theorem}

\begin{proof}[Proof and Discussion]
This was shown for very affine varieties in \cite{Huh14}. It
was adapted to statistical models in \cite{DMS}.
These articles offer precise statements via
Horn uniformization
for $A$-discriminants \cite{GKZ}, i.e.~hypersurfaces dual to  toric varieties.
See also \cite[Corollary 3.12]{HS}.
\end{proof}

This section concludes with a connection to
scattering amplitudes in particle physics that was
discovered recently in \cite{ST}.
We consider the \emph{CEGM model}, due to Cachazo and his collaborators \cite{CEGM, CUZ}.
The role of the data vector $u$ is played by the
Mandelstam invariants. This theory rests on
the space $X^o$ of $m$ labeled points in general position in $\PP^{k-1}$, up to projective
transformations. Consider the action of the 
torus $(\CC^*)^m$ on the  Grassmannian 
$\text{Gr}(k,m) \subset \PP^{\binom{m}{k}-1}$.
 Let $\text{Gr}(k,m)^o$ be the open Grassmannian where
all Pl\"ucker coordinates are nonzero. 
The CEGM model is the $(k-1)(m-k-1)$-dimensional manifold 
\begin{equation}
\label{eq:cegm}
 X^o \quad = \quad \text{Gr}(k,m)^o/(\CC^*)^m. 
 \end{equation}

\begin{proposition}
The variety $X^o$ is very affine, with coordinates given by the $k {\times} k$ minors~of
\setcounter{MaxMatrixCols}{20}
\begin{equation}
\label{eq:Mkm}
M_{k,m} \,\,=\,\,  \begin{small} \begin{bmatrix} 
0 \quad & \quad 0 \quad & \quad 0 \quad & \quad \dots \quad & \quad 0 \quad & \,\,(-1)^k \quad & \quad 1 \quad & \quad 1 \quad & \quad 1 \quad & \quad \dots \quad & \quad 1\\ 
0 \quad & \quad 0 \quad & \quad 0 \quad & \quad \dots \quad & \quad (-1)^{k-1}  & \,\,0 \quad & \quad 1 \quad & \quad x_{1,1} \quad & \quad x_{1,2} \quad & \quad \dots \quad & \quad x_{1,m-k-1} \\ 
\vdots \quad & \quad \vdots \quad & \quad \vdots \quad & \quad \Ddots \quad & \quad \vdots \quad & \,\, \vdots \quad & \quad \vdots \quad & \quad \vdots \quad & \quad \vdots \quad & \quad \ddots \quad & \quad \vdots\\ 
0 \quad & \,\, 0 \quad & \quad -1 \quad & \quad \dots \quad & \quad 0 \quad & \,\, 0 \quad & \quad 1 \quad & \quad x_{k-3,1} \quad & \quad x_{k-3,2} \quad & \quad \dots \quad & \quad x_{k-3,m-k-1}\\ 
0 \quad & \quad 1 \quad & \quad 0 \quad & \quad \dots \quad & \quad 0 \quad & \,\, 0 \quad & \quad 1 \quad & \quad x_{k-2,1} \quad & \quad x_{k-2,2} \quad & \quad \dots \quad & \quad x_{k-2,m-k-1}\\ 
-1 \quad & \quad 0 \quad & \quad 0 \quad & \quad \dots \quad & \quad 0 \quad & \,\, 0 \quad & \quad 1 \quad & \quad x_{k-1,1} \quad & \quad x_{k-1,2} \quad & \quad \dots \quad & \quad x_{k-1,m-k-1}\\
\end{bmatrix} \end{small}.
\end{equation}
To be precise, the coordinates on $X^o \subset (\CC^*)^{\binom{m}{k}}$ are the
non-constant minors $p_{i_1 i_2 \cdots i_k}$.
\end{proposition}

Following \cite[equation (4)]{BAKFST}, the antidiagonal matrix in the
left $k \times k$ block of $M_{k,m}$ is chosen so that each unknown  $x_{i,j}$ is
precisely equal to $p_{i_1 i_2 \cdots i_k}$ for some
$i_1 < i_2 < \cdots < i_k$. The \emph{scattering potential} for the CEGM model is
the following multivalued function on $X^o$:
\begin{equation}
\label{eq:scatpot}
 \ell_u  \quad = \quad \sum_{i_1 i_2 \cdots i_k} u_{i_1 i_2 \cdots i_k} \cdot 
\text{log}(p_{i_1i_2 \cdots i_k}). 
\end{equation}
The critical point equations, known as \emph{scattering equations} \cite[equation (7)]{BAKFST}, are given by
\begin{equation}\label{eq:scattering} 
\qquad \frac{\partial \ell_u}{ \partial x_{i,j}} \,=\, 0 \qquad \text{ for }\, 1\leq i \leq k-1\, \,
\text{and} \,\, 1\leq j \leq m-k-1 .
\end{equation}
These are equations of rational functions.
Solving these equations is the agenda in \cite{CEGM, CUZ, ST}.
 
 \begin{corollary}
 The number of complex solutions to (\ref{eq:scattering}) is the ML
 degree of the CEGM model $X^o$. This number equals the signed Euler
 characteristic $(-1)^{(k-1)(m-k-1)} \cdot \chi( X^o)$.
 \end{corollary}

\begin{example}[$k=2,m=6$]
The very affine threefold $X^o$ is
embedded in $(\CC^*)^9$ via
$$ \begin{matrix} p_{24} = x_1,\,
p_{25} = x_2,\,
p_{26} = x_3,\,
p_{34} = x_1-1,\,
p_{35} = x_2-1,\, \\ 
p_{36} = x_3-1,\,
p_{45} = x_2-x_1,\,
p_{46} = x_3-x_1,\,
p_{56} = x_3-x_2 .\end{matrix}
$$
These nine coordinates on
$X^o \subset (\CC^*)^9$ are the non-constant $2 \times 2$ minors of our matrix
$$
M_{2,6} \,\,\, =\,\,\, \begin{bmatrix}
\, \,\,\,0 \quad &
 \quad 1 \quad & 
  \quad 1 \quad &
   \quad 1 \quad &
    \quad 1 \quad &
     \quad 1 \quad \\
    \,   -1 \quad &
       \quad 0 \quad &
        \quad 1 \quad &
 \quad x_1 \quad &        
 \quad x_2 \quad &        
 \quad x_3 \quad 
      \end{bmatrix}.
   $$
   The scattering potential is the analogue to the log-likelihood function in statistics:
$$
\ell_u \,\, = \,\,   
u_{24} \,\text{log}(p_{24}) +    
u_{25} \,\text{log}(p_{25}) + \cdots +
u_{56} \,\text{log}(p_{56}) .
$$
This function has six critical points in $X^o$.
Hence $ \text{MLdegree}(X^o) = -\chi(X^o)= 6$.
\end{example}

We now examine the
number of critical points of the scattering potential
(\ref{eq:scatpot}).

\begin{theorem} \label{thm:CEGM}
The known values of the ML degree for
the CEGM model (\ref{eq:cegm}) are as follows.
For $k=2$, the ML degree equals $(m-3)!$
for all $m \geq 4$.
For $k=3$, it equals
$2,26,1272,188112, 74570400$ 
for $m=5,6,7,8,9$, and for
 $k=4,m=8$ it equals $5211816$.
\end{theorem}

\begin{proof}
We refer to \cite[Example 2.2]{BAKFST},
\cite[Theorem 5.1]{BAKFST} and
\cite[Theorem 6.1]{BAKFST} for $k=2,3,4$.
\end{proof}

Knowing these ML degrees helps in solving
the scattering equations reliably.  We demonstrated
in \cite{BAKFST, ST} how this can be done in practice with
\texttt{HomotopyContinuation.jl} \cite{HCjl, BRT}.
For instance, we see in \cite[Table 1]{ST} that
the $10! = 3628800$ solutions for $k=2,m=13$
are found in under one hour.
See \cite[Section 6]{BAKFST} for the solution in
the challenging case $k=4,m=8$.

\section{Nonlinear Algebra meets Linear PDE}
\label{sec4}

In his 1938 article on foundations of algebraic geometry,
Wolfgang Gr\"obner introduced differential operators to characterize
membership in a polynomial ideal. He solved this for zero-dimensional
ideals using Macaulay's inverse systems \cite{Gro}. Gr\"obner wanted this for
all ideals, ideally with algorithmic methods. This was finally achieved
in the article~\cite{CS}.

Analysts made substantial contributions to this subject.
In the 1960's, Leon Ehrenpreis and Victor Palamodov studied solutions
to linear partial differential equations (PDE) with constant coefficients.
A main step was the characterization of membership in a primary
ideal by Noetherian operators. This led to their celebrated {\em Fundamental Principle}.
That result is presented in Theorem \ref{thm:EP}.
For background reading see \cite{AHS, CHS, CS}
and their references.

\begin{example}[$n=3$] We give an illustration by
exploring a progression of four questions.

\noindent \emph{Question 1}:
What are the solutions to the system of equations
$\,x_1^2 = x_2^2 = x_1 x_3 - x_2 x_3^2 = 0\,$?

\noindent \emph{Question 2}: Determine all functions
$\phi(z_1,z_2,z_3)$ that satisfy the following three linear PDE:
$$
 \frac{\partial^2 \phi}{\partial z_1^2} \,=\,
 \frac{\partial^2 \phi}{\partial z_2^2} \,=\,
  \frac{\partial^2 \phi}{\partial z_1 \partial z_3} -
 \frac{\partial^3 \phi}{\partial z_2 \partial z_3^2} \,=\, 0.
$$

\noindent \emph{Question 3}: Which polynomials lie in the ideal
\begin{equation}
\label{eq:I222}
 I \,\, = \,\, \langle x_1^2, x_2^2, x_1 - x_2 x_3 \rangle\,\, \cap \,\,
\langle   x_1^2, x_2^2, x_3 \rangle \, ? \end{equation}

\noindent \emph{Question 4}:
Describe the geometry of
the subscheme $V(I)$ of affine $3$-space given by (\ref{eq:I222}).

Here are our answers to these four questions. Notice how they are intertwined:

\noindent \emph{Answer 1}:
Assuming that $x_i^2 = 0$ implies $x_i = 0$, the equations are equivalent to
 $x_1 = x_2 = 0$. Their solution set is a line through the origin in $3$-space, namely
 the $x_3$-axis.

\noindent \emph{Answer 2}:
The solutions to these PDE are precisely the functions $\phi(z)$ that have the form
\begin{equation}
\label{eq:havetheform}
 \phi(z_1,z_2,z_3) \,\, = \,\, \xi(z_3) \, + \, \bigl(z_2 \psi(z_3) + z_1 \psi'(z_3) \bigr) \,+\,
\alpha z_1 z_2 \,+ \, \beta z_1, 
\end{equation}
where $\alpha,\beta$ are constants, and
$\xi$ and $\psi$ are differentiable functions in one variable.

\noindent \emph{Answer 3}:
A polynomial $f$ is in the ideal $I$ if and only if the following four conditions hold:
Both $f$ and $\frac{\partial f}{\partial x_2} + x_3 \frac{\partial f}{\partial x_1}$ vanish
on the $x_3$-axis, and both $\frac{\partial^2 f}{\partial x_1 x_2}$ and
$\frac{\partial f}{\partial x_1}$ vanish at the origin.

\noindent \emph{Answer 4}: This scheme is a double $x_3$-axis together with
an embedded point of length two at the origin.
Hence $I$ has arithmetic multiplicity four:
two for the line and two for the point.

Answer 4 reveals the multiplicity structure on the naive solution set in Answer 1.
This is characterized by four features, one for each differential condition
in Answer 3. These are in natural bijection with the four summands
of the general solution (\ref{eq:havetheform}) in Answer 2.
\end{example}

We now turn to ideals $I$ in the polynomial ring $\CC[x] = \CC[x_1,\ldots,x_n]$.
We identify the $n$ variables with differential operators $x_i = \partial_{z_i} $
that act on functions $\phi(z) = \phi(z_1,\ldots,z_n)$.
In this manner, each $I$ is a system of linear homogeneous PDE
 with constant coefficients.
This role of polynomials
 is the topic of Section 3.3 in the textbook~\cite{MS}.
The story begins in \cite[Lemma 3.25]{MS} with the following
 encoding of the variety $V(I)$ in the solutions to the PDE.

\begin{lemma}
A point $a \in \CC^n$ lies in the variety $V(I)$ if and only if the
exponential function $\text{exp}(a \cdot z) = \text{exp}(a_1 z_1 + \cdots + a_n z_n)$
is a solution to the system of linear PDE given by $I$.
\end{lemma}

Since our PDE are linear, their solution sets are linear spaces.
Arbitrary $\CC$-linear combinations of solutions are again solutions.
The following proposition makes this precise.

\begin{proposition} \label{prop:fourthree}
Given any measure $\mu$ on the variety $V(I)$, here is a solution to our PDE:
\begin{equation}
\label{eq:expintegral}
\phi(z) \,\, = \,\, \int_{V(I)} \text{exp}( a \cdot z)\, d\mu(a).
\end{equation}
If $I$ is a prime ideal then every solution to the PDE admits such an integral representation.
\end{proposition}

The first part of Proposition \ref{prop:fourthree} is straightforward.
Recall that an ideal $Q$ is \emph{primary} if it has only one associated prime~$P$.
The second part is a special~case of the following result.

\begin{theorem}[Ehrenpreis-Palamodov]
\label{thm:EP}
Fix a prime ideal $P$ in $\CC[x]$. For any $P$-primary ideal $Q$ in $\CC[x]$, there
exist polynomials $B_1,\ldots,B_m\,$ in $\,2n$ unknowns such that
the function
\begin{equation}
\label{eq:PDintegral}
\phi(z) \quad = \quad \sum_{i=1}^m \int_{V(P)} \!\! B_i(x,z)\, \text{exp}(x \cdot z) \,d \mu_i(x)
\end{equation}
is a solution to the PDE given by $Q$, for any measures $\mu_1,\ldots,\mu_m$ on the variety $V(P)$. 
Conversely, every solution
$\phi(z)$ of the PDE given by $Q$ admits such an integral representation.
\end{theorem}

\begin{proof} See \cite[Theorem 3.3]{CHS}
and the pointers to the analysis literature given there.
\end{proof}

The polynomials $B_1(x,z),\ldots,B_m(x,z)$ are known as \emph{Noetherian multipliers}.
They depend only on the primary ideal $Q$, and not on the function $\phi(z)$.
They encode the scheme structure imposed by $Q$ on the irreducible variety $V(P)$.
The Noetherian multipliers furnish a finite representation 
 of a vector space that is usually infinite-dimensional,
namely the space of all solutions to the PDE, within a 
suitable class of scalar-valued functions on $n$-space.

\begin{example}$[n=3]$  \label{ex:Q3a}
\, Let $Q = \langle x_1^2, x_2^2, x_1 - x_2 x_3 \rangle$ be the 
first primary ideal in (\ref{eq:I222}). Here $m=2$,
$B_1 = 1$, and $B_2 = x_3 z_1 + z_2$.
Solutions to $Q$ are given by the two summands in (\ref{eq:PDintegral}):
$$ \phi_1(z) \,\,=\,\, \int  1 \cdot \,\text{exp}(0 z_1+ 0 z_2+x_3 z_3) \,d\mu_1(x)\,\, =\,\, \xi(z_3)  \qquad \qquad \quad \text{and} \qquad $$
$$ \begin{matrix}
 \phi_2(z) \phantom{do}
&=&  \int ( z_2 + z_1 x_3 )\cdot  \text{exp}(0 z_1+0 z_2+x_3 z_3) \,d\mu_2(x) \qquad \qquad \qquad \\
&=&  \quad z_2   \int \! \text{exp}( 0 z_1{+} 0 z_2{+} x_3 z_3) d\mu_2(x)
    + z_1 \!  \int \! x_3 \,\text{exp}( 0 z_1{+} 0 z_2{+}x_3 z_3) d\mu_2(x) \\
&=&    z_2 \, \psi(z_3) \,+\, z_1\, \psi'(z_3) . \quad
\end{matrix}
$$
We conclude that our solution $\phi_1(z) + \phi_2(z)$ agrees with the first two summands in
(\ref{eq:havetheform}).
\end{example}

Switching the roles of $x$ and $z$, we now set
$z_1 = \partial_{x_1},\ldots,z_n = \partial_{x_n}$ in  the Noetherian multipliers.
Here it is important that the $x$-variables
occur to the left of the $z$-variables in the monomial expansion of each
$B_i(x,z)$.
This results in the \emph{Noetherian operators}
$ B_i(x,\partial_x)$. These operators
are elements in the Weyl algebra and they act on
polynomials in $\CC[x]$.
We use $\bullet$ to denote the action
of differential operators on polynomials and other functions.

\begin{proposition} 
The Noetherian operators determine membership in the primary ideal $Q$.
Namely, a polynomial $f(x)$ lies in $Q$
if and only if $B_i(x,\partial_x) \bullet f(x)$ lies in $P$
for $i=1,\ldots,m$.
\end{proposition}

\begin{proof}
This is the content of \cite[Proposition 4.8]{AHS}.
See also \cite[Theorems 3.2 and 3.3]{CHS}.
\end{proof}

\begin{example} \label{ex:Q3b}
From $B_1$ and $B_2$ in Example \ref{ex:Q3a},
we obtain the Noetherian operators
$1$ and $ x_3 \partial_{x_1} {+} \partial_{x_2}$.
A polynomial $f$ lies in $Q$ if and only if
 $f$ and $( x_3 \partial_{x_1} {+} \partial_{x_2}) \bullet f$
are  in $P=\langle x_1,x_2 \rangle$.
\end{example}

We have seen that Noetherian multipliers and Noetherian
operators are two sides of the same coin. While the latter
characterize the membership in a primary ideal, 
as envisioned by Gr\"obner \cite{Gro}, the former
furnish the general solution to the associated PDE.
A next step is the extension from primary to arbitrary
ideals in  the polynomial ring $R=\CC[x]$.
To be more general,  we consider
an arbitrary submodule $M$ of
the free module $R^k$. Such a submodule
represents a system of linear PDE as before, but for 
vector-valued functions $\phi : \CC^n \rightarrow \CC^k$.

For a vector $m \in R^k$, the quotient
$(M:m)$ is the ideal $\{f \in R: fm \in M\}$.
 A prime ideal $P_i \subseteq R$ is
{\em associated to} the module $M$ 
if $(M:m) = P_i$ for some $m \in R^k$. The list of all
associated primes of $M$ is finite, say $P_1,\ldots,P_s$.
If $s=1$ then $M$ is $P_1$-primary.
A {\em primary decomposition} of $M$ is a list of primary submodules
$M_1,\ldots,M_s \subseteq R^k$
where $M_i$ is $P_i$-primary and
$\, M  =  M_1 \cap M_2 \cap \cdots \cap M_s $.
The contribution of the primary module $M_i$ to $M$
is quantified by a positive integer $m_i$, called
  the  arithmetic length of $M$ along $P_i$.
To define this, we consider the localization
$(R_{P_i})^k/M_{P_i}$. This is a module over
the local ring $R_{P_i}$. The {\em arithmetic length}
is the length of the largest submodule of finite length in
$(R_{P_i})^k/M_{P_i}$.
The sum $m_1 + \cdots + m_s$ is  denoted $\text{amult}(M)$
and called the {\em arithmetic multiplicity} of $M$.

\begin{example}[$n=3,k=1$]
The ideal $I$ in (\ref{eq:I222}) has
arithmetic multiplicity $4$.
The arithmetic length is $m_1 = m_2 = 2$
along each of the
associated primes
$P_1 = \langle x_1,x_2\rangle$ and 
$P_2 = \langle x_1,x_2,x_3 \rangle$.
\end{example}

We now present an extension of 
Theorem \ref{thm:EP} to PDE for
vector-valued functions.
Let $V_i = V(P_i) \subset \CC^n$ be the irreducible
variety defined by the $i$th associated prime $P_i$ of  $M$.

\begin{theorem}[Ehrenpreis-Palamodov for modules]
\label{thm:EP2}
For any submodule $M \subset R^k$, there exist
$\,\text{amult}(M) = \sum_{i=1}^s m_i \,$ Noetherian multipliers:
these are vectors $B_{ij} \in \CC[x,z]^k$ such that 
  \begin{equation}
\label{eq:anysolution}
	\phi(z) \,\,\,= \,\,\, \sum_{i=1}^s \sum_{j=1}^{m_i} \,\int_{V_i} \!\! B_{ij} \! 
	\left(x,z\right)
	\exp\left( x \cdot z \right) d\mu_{ij} (x)
\end{equation}
is a solution to the PDE given by $M$.
Here $\mu_{ij}$ are measures that are supported on the  variety~$V_i $.
Conversely, every solution to that PDE admits such an integral representation.
\end{theorem}

\begin{proof}
This statement appears in \cite[Theorem 2.2]{AHS}.
Differential primary decomposition 
\cite[Theorem 4.6 (i)]{CS} shows that
the number of inner summands equals the
arithmetic length $m_i$.
\end{proof}

As before, we can pass from Noetherian multipliers
$B_{ij}(x,z)$ to Noetherian operators $B_{ij}(x,\partial_x)$
and obtain a differential primary decomposition of $M$;
see \cite{CS} and \cite[\S 4]{AHS}. 
We write $\bullet$ for the application of a vector of differential
operators to a vector of functions. This is done coordinatewise and 
followed by summing the coordinates.  The result is a function.

\begin{corollary} 
The Noetherian operators determine membership in the module $M$.
Namely, a vector $m \in R^k$ lies in $M$
if and only if $\,B_{ij}(x,\partial_x) \bullet m(x)\,$ vanishes on $\,V_i\,$
for all~$i,j$.
\end{corollary}

The package \texttt{NoetherianOperators} \cite{Chen}
in the software \texttt{Macaulay2} \cite{MAC2} is a convenient tool
for solving the  PDE given by a 
submodule \texttt{M} of $R^k$. 
Typing \texttt{amult(M)} gives the arithmetic multiplicity of \texttt{M}.
The command \texttt{solvePDE(M}) lists all associated primes
$P_i$ along with their Noetherian multipliers $B_{ij}(x,z)$.
These features are described in~\cite[\S 5]{AHS}.

What is intended with the command \texttt{solvePDE}
vastly generalizes the problem of solving systems of
polynomial equations, which is central to nonlinear algebra.
That point is argued in \cite[Chapter 3]{MS}, which culminates
with writing polynomials as PDE.
First steps towards a  numerical version of \texttt{solvePDE} 
are discussed in \cite[\S 7.5]{AHS} and \cite{Chen}.
It is instructive to revisit \cite[Theorem 3.27]{MS}
through the lens of  Theorem \ref{thm:EP2}.
The solution space of an ideal $I$ is finite-dimensional
if and only if each  $V_i$ is a point.
If,  furthermore, 
$s=1$ and $ V_1 = \{0\}$, then
the Noetherian multipliers
$B_1(z),\ldots,B_{m_1}(z)$ form a basis for the solution space of $I$.

If we pass from ideals to modules then even the case
$s=1, V_1 = \CC^n$ is quite rich and interesting,
especially in connection with the theory of wave cones \cite{ADHR}.
We close with a nontrivial example which shows
what wave solutions are and how they can be constructed.

\begin{example}[$n=4,k=7$]
Let $R = \CC[x]$ and let $M \subset R^7$ be the module generated by
$(x_1,x_2,x_3,x_4,0,0,0)$,
$\,(0,x_1,x_2,x_3,x_4,0,0)$,
$\,(0,0,x_1,x_2,x_3,x_4,0)$ and
$\,(0,0,0,x_1,x_2,x_3,x_4)$.
This module is  primary with $V_1 = \CC^4$
and $\text{amult}(M) = 3$.
It represents a first-order PDE for unknown functions
$\phi : \RR^4 \rightarrow \RR^7$.
To explore solutions of $M$, we apply the
\texttt{Macaulay2} command \texttt{solvePDE}.
The code outputs
three Noetherian multipliers, namely the rows of 
$$ \!\!\!
\begin{bmatrix}
 x_2^4-3 x_1 x_2^2 x_3+x_1^2 x_3^2+2 x_1^2 x_2 x_4 &  \,\,\,\,
   2 x_1^2 x_2 x_3{-}x_1 x_2^3{-}x_1^3 x_4 \quad & \quad   x_1^2 x_2^2{-}x_1^3 x_3 \quad & \quad  \! -x_1^3 x_2 \quad & \,\,\,\,  x_1^4 \quad & \,\,\,\,  0 \quad & \,\,\,\,   0 
\\     
 x_2^3 x_3{-}2 x_1 x_2 x_3^2{-}x_1 x_2^2 x_4 {+} 2 x_1^2 x_3 x_4 \quad & \,\,\,
x_1^2 x_3^2 {-}x_1 x_2^2 x_3{+}x_1^2 x_2 x_4 \quad & \,\,\, 
      x_1^2 x_2 x_3{-}x_1^3 x_4 \quad & \quad   \! -x_1^3 x_3 \quad & \,\,\,\,   0 \quad & 
      \,\,\,\,   x_1^4 \quad & \,\,\,\,   0
\\      x_2^3 x_4-2 x_1 x_2 x_3 x_4+x_1^2 x_4^2 \quad & \quad  
     -x_1 x_2^2 x_4+x_1^2 x_3 x_4 \quad & \quad   x_1^2 x_2 x_4 \quad & \quad\!   -x_1^3 x_4 \quad & \,\,\,   0 \quad & \,\,\,\,   0 \quad & \,\,\,\,  x_1^4
\end{bmatrix}.
$$
These rows are syzygies of $M$.
They span all syzygies as a vector space over the function field $\RR(x)$.
Solutions $\phi$ to the PDE can be constructed from any syzygy
by applying that differential operator to 
any function $f(z_1,z_2,z_3,z_4)$. For instance, writing
subscripts for differentiation, the first row of the matrix above gives
the following solution to our PDE $M$:
$$ \phi = (
f_{2222} - 3 f_{1223} + f_{1133} + 2 f_{1124},\,
2 f_{1123} - f_{1222} - f_{1114}, \,
f_{1122}-f_{1113} ,\, -f_{1112} ,\, f_{1111}, \,0, \, 0 \,) .$$

Next, we show how nonlinear algebra makes waves.
Consider the Hankel~matrix 
$$ H(u) \,\,  = \,\,   \begin{bmatrix}
     u_1 \quad & \quad u_2 \quad & \quad u_3 \quad & \quad u_4 \\
     u_2 \quad & \quad u_3 \quad & \quad u_4 \quad & \quad u_5 \\
     u_3 \quad & \quad u_4 \quad & \quad u_5 \quad & \quad u_6 \\
     u_4 \quad & \quad u_5 \quad & \quad u_6 \quad & \quad u_7 \end{bmatrix}. 
     $$ 
     We identify
the four entries of $x \cdot H(u) $ with the generators of $M$.
The wave cones of \cite{ADHR} are the determinantal varieties
$\{ u \in \PP^6 : \text{rank}(H(u)) \leq r \}$.
For $r = 1$, this is the rational normal curve in $\PP^6$.
For $r=2$, it is the secant variety to the curve, of dimension $3$.
For $r=3$, it is the variety of secant planes.
The latter is the quartic hypersurface 
$\{ u \in \PP^6: \text{det}(H(u)) = 0\}$.
The span of our three Noetherian multipliers furnishes a parametrization of
that hypersurface. 

Any $u \in \PP^6$ with $H(u)$ of low rank yields
 wave solutions to $M$. For an illustration,~let
$$ u \, = \, (1,2,4,8,16,32,64) .$$
Here $H(u)$ has rank $1$. Its kernel is spanned by
$2 e_1-e_2, 2e_2-e_3, 2 e_3- e_4$.
For any scalar function $\psi$ in three variables, 
we obtain a function that satisfies the PDE given by $M$, namely
$$ \phi(z) \,\,=\,\, \psi(2 z_1 - z_2, 2 z_2 - z_3, 2 z_3 - z_4)   \cdot u. $$
This vector is an example of a wave solution.
If we take $\psi$ to be the Dirac distribution 
at the origin in $\RR^3$ then $\phi$ is a distributional solution that
is supported on a line in $\RR^4$.
Characterizing such low-dimensional supports of solutions
is the objective of the article~\cite{ADHR}.
\end{example}

\begin{ack}
Many thanks to Simon Telen for the computation in Example \ref{ex:simon}.
Helpful comments on draft versions of this paper were provided by
Yulia Alexandr, Claudia Fevola, Marc H\"ark\"onen, Yelena Mandelshtam,  Chiara Meroni and Charles Wang.
\end{ack}


\end{document}